\documentclass[brazil, 12pt]{amsart}
\ProvidesLanguage{portuges}
\usepackage[latin1]{inputenc}
\usepackage{indentfirst}
\usepackage{amsfonts}
\usepackage{graphicx}
\usepackage{amsmath}
\usepackage{amssymb}
\usepackage{latexsym}
\usepackage{amscd}
\usepackage[all]{xy}

\newtheorem{theorem}{Theorem}[section]

\newtheorem{lemma}[theorem]{Lemma}

\newtheorem{proposition}[theorem]{Proposition}

\newtheorem{example}[theorem]{Example}

\begin{document}

\title[Two-point AG codes on the GK maximal curves]{Two-point AG codes on the GK maximal curves}

\author{A. Sepulveda and G. Tizziotti}

\maketitle

\begin{abstract}
We determine de Weierstrass semigroup of a pair of certain rational points on the GK-curves. We use this semigroup to obtain two-point AG codes with better parameters than comparable one-point AG codes arising from these curves. These parameters are new records in the MinT's tables.
\end{abstract}

\medskip


\section{Introduction}

V.D. Goppa, \cite{goppa1} and \cite{goppa2}, constructed error-correcting codes using algebraic curves, the called \emph{algebraic geometric codes} (AG codes). The introduction of methods from algebraic geometry to construct good linear codes was one of the major developments in the theory of error-correcting codes. From that moment, in the early 1980s, many studies have emerged and the theory of Weierstrass semigroup is an important part in the study of AG codes. Its use comes from the theory of one-point codes, where there exist close connections between the parameters of one-point codes and its dual with the Weierstrass semigroup over one point on the curve, see for example \cite{vanlint}. Later these results were extended to codes and semigroups over two or more points. In \cite{gretchen}, Matthews proved for arbitrary curves that the Weierstrass gap set at a pair of points may be exploited to define a code with minimum distance greater than the Goppa bound. Despite the great interest of these codes, its utility is limited by the difficult of computing the Weierstrass semigroup at two points. In this sense, two-points codes over specific curves has been studied. In particular, for the maximal curves: Hermitian curve, Duursma and Kirov in \cite{Duursma}; Suzuki curve, Matthews in \cite{gretchen2}; and $y^q+y=x^{q^r+1}$ curve, Sepúlveda and Tizziotti in \cite{ST}.

In this work, we focus our attention on the $GK$ curves, which are maximal curves construct by Giulietti and Korchmáros over $\mathbb{F}_{q^6}$ which cannot be covered by the Hermitian curve whenever $q>2$. In \cite{GKcodes}, Fanali e Giulietti have investigated one-point AG codes over $GK$ curves and found linear codes with better parameters with respect those known previously. Here we determinate the Weierstrass semigroup $H(P_1,P_2)$ at certain two points on the $GK$ and we user this semigroup to construct two-point AG codes with better parameters than comparable one-point AG codes. Furthermore, these parameters are new records in the MinT's tables \cite{MinT}.

This work is organized as follows. In the Section 2 we introduce some basic facts about Weierstrass semigroup, AG codes and the $GK$ curves. In the Section 3, we determine the Weierstrass semigroup of a pair of points on the curve $GK$. Finally, in Section 4 we use results of the previous section to construct two-point AG codes which parameters are new records.

\section{Preliminaries}

\subsection{Weierstrass semigroup}

Let $\mathcal{X}$ be a non-singular, projective, irreducible, algebraic curve of genus $g \geq 1$ over a finite field $\mathbb{F}_{q}$ and $\mathbb{F}_{q}(\mathcal{X})$ be the field of rational functions on $\mathcal{X}$. Let $P$ be a rational point on $\mathcal{X}$ and $\mathbb{N}_{0}$ be the set of nonnegative integers. The set
$$
H(P):= \{n \in \mathbb{N}_{0} \mbox{ ; } \exists f \in \mathbb{F}_{q}(\mathcal{X}) \mbox{ with } (f)_{\infty} = n P \}
$$
where $(f)_{\infty}$ denotes the divisor of poles of $f$, is a semigroup, called the \textit{Weierstrass semigroup} of $\mathcal{X}$ at $P$. The set $G(P) = \mathbb{N}_{0} \setminus H(P)$ is called \textit{Weierstrass gap set} of $P$ and its cardinality is exactly
 $g$. In the case of two distinct rational points $P_{1}$ and $P_{2}$ on $\mathcal{X}$ we have the set
$$
H(P_{1}, P_{2}) = \{(n_{1}, n_{2}) \in \mathbb{N}_{0}^2 \mbox{ ; } \exists f \in \mathbb{F}_{q}(\mathcal{X}) \mbox{ with } (f)_{\infty} = n_{1} P_{1} + n_{2} P_{2} \};
$$

that is, the Weierstrass semigroup of $\mathcal{X}$ at $P_{1}$ and $P_{2}$. Analogously, the set $G(P_{1}, P_{2}) = \mathbb{N}_{0}^2 \setminus H(P_{1}, P_{2})$ is called the  Weierstrass gap set of the pair $(P_{1}, P_{2})$. Unlike the one-point case, the cardinality of $G(P_{1}, P_{2})$ depends of the choice of points $P_{1}$ and $P_{2}$ (see \cite{kim}). The study of Weierstrass semigroup of a pair of points was initiated by Arbarello et al., in \cite{arbarello}. Homma, in \cite{homma2}, found bounds for the cardinality of $G(P_{1}, P_{2})$, and discovered a connection between $H(P_{1}, P_{2})$ and a permutation of the set $\{1,2, \ldots, g\}$. 

Now we give some concepts that are important in this work. Let $P_{1}$ and $P_{2}$ be rational points on $\mathcal{X}$. We define $\beta_{\alpha} := min \{ \beta \in \mathbb{N}_{0} \mbox{ ; } (\alpha, \beta ) \in H(P_{1}, P_{2}) \}$ and we have that $\{ \beta_{\alpha} \mbox{ ; } \alpha \in G(P_{1}) \} = G(P_{2})$, see Lemma 2.6 in \cite{kim}. If $\alpha_{1} < \alpha_{2} < \cdots < \alpha_{g}$ and $\beta_{1} < \beta_{2} < \cdots < \beta_{g}$ are the gaps sequences at $P_{1}$ and $P_{2}$, respectively, then the above equality implies that there exist a one-to-one correspondence between $G(P_{1})$ and $G(P_{2})$. So there exists a permutation $\sigma$ of the set $\{1,2, \ldots , g\}$ such that $\beta_{\alpha_{i}} = \beta_{\sigma(i)}$. This permutation is denoted by $\sigma(P_{1}, P_{2})$. The graph of the bijective map between $G(P_{1})$ and $G(P_{2})$, denoted by $\Gamma (P_{1}, P_{2})$, is the set
$$
\Gamma(P_{1}, P_{2})  : = \{ (\alpha_{i} , \beta_{\alpha_{i}}) \mbox{ ; }  i=1,2, \ldots ,g \} =  \{ (\alpha_{i} , \beta_{\sigma(i)}) \mbox{ ; }  i=1,2, \ldots ,g \}.
$$

\begin{lemma} \label{lemma 1} \cite[Lemma 2]{homma2}
Let $\Gamma '$ be a subset of $(G(P_{1}) \times G(P_{2})) \cap H(P_{1},P_{2})$. If there exists a permutation $\tau$ of $\{ 1, 2, \ldots , g\}$ such that $\Gamma ' = \{ (\alpha_{i} , \beta_{\tau(i)}) \mbox{ ; }  i=1,2, \ldots ,g \}$, then $\Gamma ' = \Gamma(P_{1}, P_{2})$.
\end{lemma}

\medskip

Compute $\Gamma (P_{1}, P_{2})$ allows to compute $H(P_{1}, P_{2})$ as follows. Given $\mathbf{x} = (\alpha_{1}, \beta_{1})$, $\mathbf{y} = (\alpha_{2}, \beta_{2}) \in \mathbb{N}_{0}^2$, the \textit{last upper bound} (or lub) of $\mathbf{x}$ and $\mathbf{y}$ is defined as $lub(\mathbf{x},\mathbf{y}):= (max\{\alpha_{1}, \alpha_{2}\}, max\{\beta_{1}, \beta_{2}\})$.

In \cite{kim}, we see that if $\mathbf{x},\mathbf{y} \in H(P_{1}, P_{2})$, then $lub(\mathbf{x},\mathbf{y}) \in H(P_{1}, P_{2})$. Moreover, we have the following results.

\medskip

\begin{lemma} \label{lemma 2.1 kim} \cite[Lemma 2.1]{kim}
Let $P_{1}$ and $P_{2}$ be two distinct rational points. For $(\alpha_1, \alpha_2) \in \mathbb{N}_{0}^2$ the following are equivalents:

a) $(\alpha_1, \alpha_2) \in H(P_1, P_2)$;

b) $\ell (\alpha_1 P_1 + \alpha_2 P_2) = \ell ((\alpha_1 -1)P_1 + \alpha_2 P_2) +1= \ell (\alpha_1 P_1 + (\alpha_2 -1)P_2) + 1$.
\end{lemma}

\medskip

Note that if $(\alpha_1, \alpha_2) \in G(P_1, P_2)$, then by lemma above either $\ell (\alpha_1 P_1 + \alpha_2 P_2) = \ell ((\alpha_1 -1)P_1 + \alpha_2 P_2) $ or $\ell (\alpha_1 P_1 + \alpha_2 P_2) = \ell (\alpha_1P_1 + (\alpha_2 -1) P_2) $.

\medskip

\begin{lemma} \label{lemma 2} \cite[Lemma 2.2]{kim}
Let $P_{1}$ and $P_{2}$ be two distinct rational points. Then $H(P_{1},P_{2}) = \{ lub (\mathbf{x},\mathbf{y}) \mbox{ : } \mathbf{x},\mathbf{y} \in \Gamma(P_{1}, P_{2}) \cup (H(P_{1}) \times \{0\}) \cup (\{0\} \cup H(P_{2})) \}$.
\end{lemma}

\medskip

Then, by the lemma above for obtain the Weierstrass semigroup $H(P_{1},P_{2})$ is sufficient determine $\Gamma(P_{1}, P_{2})$. In this sense the set $\Gamma(P_1,P_2)$ is called \emph{minimal generating} of $H(P_1,P_2)$. For more details about Weierstrass semigroups theory see e.g. \cite{ballico} and \cite{carvalho}.

\subsection{AG codes}

Let $\mathcal{X}$ be as above, $\mathbb{F}_{q}(\mathcal{X})$ be the field of rational functions on $\mathcal{X}$ and $\Omega(\mathcal{X})$ be the space of differentials forms on $\mathcal{X}$.  For a divisor $G$ on $\mathcal{X}$, we consider the vector spaces $L(G):= \{ f \in \mathbb{F}_{q}(\mathcal{X}) \mbox{ ; } (f) + G \geq 0 \} \cup \{ 0 \}$ and $ \Omega(G) = \{ \omega \in \Omega(\mathcal{X}) \mbox{ ; } (\omega) \succeq G \} \cup \{0\}$, where $(f)$ and $(\omega)$ are the divisors of $f$ and $\omega$, respectively. Let $D=P_1 + \ldots + P_n$ be a divisor on $\mathcal{X}$ such that $P_i \neq P_j$ for $i \neq j$ and $supp(D) \cap supp(G) = \emptyset$. The AG codes $C_{L}(D,G)$ and $C_{\Omega}(D,G)$ are defined by

$$
C_{L}(D,G):= \{ (f(P_{1}), \ldots , f(P_{n})) \mbox{ ; } f \in L(G)\};
$$
$$
C_{\Omega}(D,G):= \{ (res_{P_{1}}(\omega), \ldots , res_{P_{n}}(\omega)) \mbox{ ; } \omega \in \Omega (G-D) \}.
$$

The AG codes $C_{L}(D,G)$ and $C_{\Omega}(D,G)$ are dual to each other. Is usual denote the AG code $C_{L}(D,G)$ simply by $C(D,G)$. The Riemann-Roch theorem makes it possible to estimate the parameters, length $n$, dimension $k$ and minimum distance $d$ of AG codes. In particular, if $2g-2 < deg(G) < n$, then $C(D,G)$ has dimension $k=deg(G) - g + 1$ and minimum distance $d \geq n - deg(G)$, see [\cite{vanlint2} , Theorem 10.6.3], and $C_{\Omega}(D,G)$ has dimension $k_{\Omega}= n - deg(G) + g - 1$, and minimum distance $d_{\Omega} \geq deg(G) - 2g + 2$, see [\cite{vanlint2} , Theorem 10.6.7]. In this sense, it is natural construct codes over curves with many rational points, hence the importance of the study of codes arising from maximal curves. We remember that a curve $\mathcal{X}$ of genus $g$ over $\mathbb{F}_{q}$ is a maximal curve if its number of $\mathbb{F}_{q}$-rational points is attains the Hasse-Weil upper bound, namely equals $2g\sqrt{q}+q+1$.

If $G=aQ$ for some rational point $Q$ on $\mathcal{X}$ and $D$ is the sum of all the other rational points on $\mathcal{X}$, then the codes $C(D,G)$ and $C_{\Omega}(D,G)$ are called \textit{one-point AG codes}. Analogously, if $G= a_{1}Q_{1} + a_{2}Q_{2}$, for two distinct rational points, then $C(D,G)$ and $C_{\Omega}(D,G)$ are called \textit{two-point AG codes}. For more details about coding theory see e.g., \cite{vanlint}, \cite{stichtenoth2} and \cite{vanlint2}.

The next result relate the Weierstrass gap set of a pair of points to the minimum distance of the corresponding two-point code.

\medskip

\begin{theorem} \label{theorem gretchen} \cite[Theorem 3.1]{gretchen}
Assume that $(a_{1} , a_{2} ) \in G(P_{1} , P_{2} )$ with $a_{1} \geq 1$ and $dim(L(a_{1} P_{1} + a_{2} P_{2} ) =
dim(L((a_{1}-1) P_{1} + a_{2} P_{2} )$. Suppose $(b_{1} , b_{2} - t - 1) \in G(P_{1} , P_{2} )$ for all $t$, $0 \leq t \leq
min \{ b_{2} - 1, 2g - 1 - (a_{1} + a_{2} ) \}$. Set $G = (a_{1} + b_{1} - 1)P_{1} + (a_{2} + b_{2} - 1)P_{2}$, and let
$D = Q_{1} + \cdots + Q_{n}$, where the $Q_{i}$ are distinct rational points, each not belonging
to the support of $G$. If the dimension of $C_{\Omega}(D, G)$ is positive, then the minimum
distance of this code is at least $deg (G) - 2g + 3$.
\end{theorem}

\medskip

Given a linear code over $\mathbb{F}_{q}$ with parameters $[n,k,d]$, the following proposition shows us how to get new codes with different $n$ and $k$.

\medskip

\begin{proposition} \label{proposition s} \cite[Exercise 7, (iii)]{tsfasman}
Is there is a linear code over $\mathbb{F}_{q}$ of length $n$, dimension $k$ and minimum distance $d$, then for each nonnegative integer $s < k$, there exists a linear code over $\mathbb{F}_{q}$ of length $n-s$, dimension $k-s$ and minimum distance $d$.
\end{proposition}

\subsection{The GK curves} Let $q=n^3$, where $n \geq 2$ is a prime power. The $GK$-curve over $\mathbb{F}_{q^2}$ is the curve of $\mathbb{P}^{3}(\overline{\mathbb{F}}_{q^2})$ with affine equations

\begin{equation} \label{equation GK}
\left\{ \begin{array}{c}
  Z^{n^2-n+1} = Y h(X) \\
  X^{n} + X = Y^{n+1}\;\;\;,
\end{array} \right.
\end{equation}

where $\displaystyle h(X) = \sum_{i=0}^{n} (-1)^{i+1} X^{i(n-1)}$. We will denote this curve simply by $GK$. The curve $GK$ is absolutely irreducible, nonsingular, has $n^8 - n^6 + n^5 + 1$ $\mathbb{F}_{q^2}$-rational points, a single point at infinity $P_{\infty}=(1:0:0:0)$ and its genus is $g = \frac{1}{2} (n^3 + 1)(n^2 - 1) + 1$. The $GK$ curve has an important properties as it lies on the Hermitian surface $\mathcal{H}_3$ with affine equation $X^{n^3} + X = Y^{n^3 + 1} + Z^{n^3+1}$; it is a maximal curve and, for $q>8$, $GK$ is the only know curve that is maximal but not $\mathbb{F}_{q^2}$-covered by the Hermitian curve $\mathcal{H}_2$ defined over $\mathbb{F}_{q^2}$ and its automorphism group $Aut(GK)$ has size $n^3(n^3+1)(n^2-1)(n^2-n+1)$ which turns out to be very large compared to the genus $g$.

Let $GK(\mathbb{F}_{q^2})$ be the set of $\mathbb{F}_{q^2}$-rational points of $GK$. We will denote a rational point $P=(a,b,c) \in GK(\mathbb{F}_{q^2})$ by $P_{(a,b,c)}$ whereas $P_0=(0,0,0)$. Since $P_{\infty}$ is the unique infinite point of $GK$ and the function field $\mathbb{F}_{q^2}(GK)$ is $\mathbb{F}_{q^2}(x,y,z)$ with $z^{n^2-n+1} = y h(x)$ and $x^n + x = y^{n+1}$ we have the next proposition.

\medskip

\begin{proposition}\label{divisors}
Let $x,y,z\in \mathbb{F}_{q^2}(x,y,z)$. Then\\

\noindent$(x) = (n^3 + 1)P_0 - (n^3+1)P_{\infty} ;$

\smallskip

\noindent$(y) = \displaystyle\sum_{i=1}^{n}(n^2 - n + 1)P_{(a_i,0,0)} - n(n^2-n+1)P_{\infty} \mbox{ ; } a_{i}^{n}+a_i=0 ;$

\smallskip

\noindent$(z) = \displaystyle\sum_{i=1}^{n^3}P_{(a_i,b_i,0)} - n^3 P_{\infty} \mbox{ ; } a_{i}^{n}+a_1=b_{i}^{n+1}$ $ \mbox{ with } a_i,b_i \in \mathbb{F}_{n^2}, \mbox{ } \forall i = 1, \ldots, n^3.$
\end{proposition}

\medskip

The next proposition give us the Weierstrass semigroup at certain points on $GK$ and the following theorem assures us that such points are in the same orbit.

\medskip

\begin{proposition} \label{proposition H(P)} \cite[Proposition 3.1]{GKcodes}
$H(P_{\infty}) = H(P_{(a,b,0)}) = \langle n^3 - n^2 + n , n^3 , n^3 + 1 \rangle$.
\end{proposition}

\medskip

\begin{theorem} \label{theorem bitransitive} \cite[Theorem 3.4]{GKcodes}
The set of $\mathbb{F}_{q^2}$-rational points of $GK$ splits into two orbits under the action of $Aut(GK)$. One orbit, say $\mathcal{O}_1$, has size $n^3+1$ and consists of the points $P_{(a,b,0)} \in GK(\mathbb{F}_{q^2})$ together with the infinite point $P_{\infty}$. The other orbit has size $n^3(n^3+1)(n^2-1)$ and consists of the points $P_{(a,b,c)} \in GK(\mathbb{F}_{q^2})$ with $c \neq 0$. Furthermore, $Aut(GK)$ acts on $\mathcal{O}_1$ as $PGU(3,n)$ in its doubly transitive permutation representation.
\end{theorem}

\medskip

For more details about this curve, see \cite{GK}.

\section{The Weierstrass semigroup $H(P_1,P_2)$ at a certain pairs of points on $GK$}

In this section we will determine the Weierstrass semigroup $H(P_1, P_2)$ for certain pairs of points on the curve $GK$. We will concentrate our results in the case $P_1=P_0$ and $P_2 = P_{\infty}$ but, by Theorem \ref{theorem bitransitive}, the results also continue to be valid for any points $P_{(a,b,0)} \in GK(\mathbb{F}_{q^2})$. That is, we can exchange the points $P_0$ and $P_{\infty}$ for any points on the orbit $\mathcal{O}_1$ given in the Theorem \ref{theorem bitransitive}.

First, let's consider the case where $n=2$ in Equation (\ref{equation GK}), that is, consider the curve $GK$ with affine equations

\begin{equation} \label{equation GK n 2}
\left\{ \begin{array}{c}
  Z^{3} = Y (1+X+X^2) \\
  X^{2} + X = Y^{3}
\end{array} \right.
\end{equation}

In this case, the genus $g=10$ and, by Proposition \ref{proposition H(P)}, $H(P_0)=H(P_{\infty}) = \langle 6,8,9 \rangle$ and then $G(P_{0}) = G(P_{\infty}) = \{ 1,2,3,4,5,7,10,11,13,19 \}$. Let
$$
T= \left \{ \dfrac{y^2}{x}, \dfrac{yz}{x}, \dfrac{z^2}{x}, \dfrac{y^2z}{x}, \dfrac{yz^2}{x}, \dfrac{y^2z^2}{x}, \dfrac{y^2z}{x^2}, \dfrac{yz^2}{x^2}, \dfrac{y^2z^2}{x^2}, \dfrac{y^2z^2}{x^3} \right \}.
$$

Then $|T|=10$ and we have

\medskip

$\bullet$ $\left ( \dfrac{y^2}{x} \right )_{\infty} = 3P_{0} + 3P_{\infty}$

\medskip

$\bullet$ $\left ( \dfrac{yz}{x} \right )_{\infty} = 5P_{0} + 5P_{\infty}$

\medskip

$\bullet$ $\left ( \dfrac{z^2}{x} \right )_{\infty} = 7P_{0} + 7P_{\infty}$

\medskip

$\bullet$ $\left ( \dfrac{y^2z}{x} \right )_{\infty} = 2P_{0} + 11P_{\infty}$

\medskip

$\bullet$ $\left ( \dfrac{yz^2}{x} \right )_{\infty} = 4P_{0} + 13P_{\infty}$

\medskip

$\bullet$ $\left ( \dfrac{y^2z^2}{x} \right )_{\infty} = P_{0} + 19P_{\infty}$

\medskip

$\bullet$ $\left ( \dfrac{y^2z}{x^2} \right )_{\infty} = 11P_{0} + 2P_{\infty}$

\medskip

$\bullet$ $\left ( \dfrac{yz^2}{x^2} \right )_{\infty} = 13P_{0} + 4P_{\infty}$

\medskip

$\bullet$ $\left ( \dfrac{y^2z^2}{x^2} \right )_{\infty} = 10P_{0} + 10P_{\infty}$

\medskip

$\bullet$ $\left ( \dfrac{y^2z^2}{x^3} \right )_{\infty} = 19P_{0} + P_{\infty}$

\medskip

Let $\Gamma ' = \{ (3,3), (5,5), (7,7), (2,11), (4,13), (1,19), $ $(11,2), (13,4), (10,10), (19,1) \}$. Then, $\Gamma ' \subset G(P_0) \times G(P_{\infty})$ and by Lemma \ref{lemma 1} follows that $\Gamma ' = \Gamma(P_0, P_{\infty})$.

Therefore, for $n=2$ we get the Weierstrass semigroup $H(P_0, P_{\infty})$. The Figure 1 depicts $H(P_0, P_\infty)\cap A^2$, where $A$ denotes the set of non-negative integers less than $2g+1$.

\begin{figure}[h!b]
\centering
\includegraphics[scale=0.9]{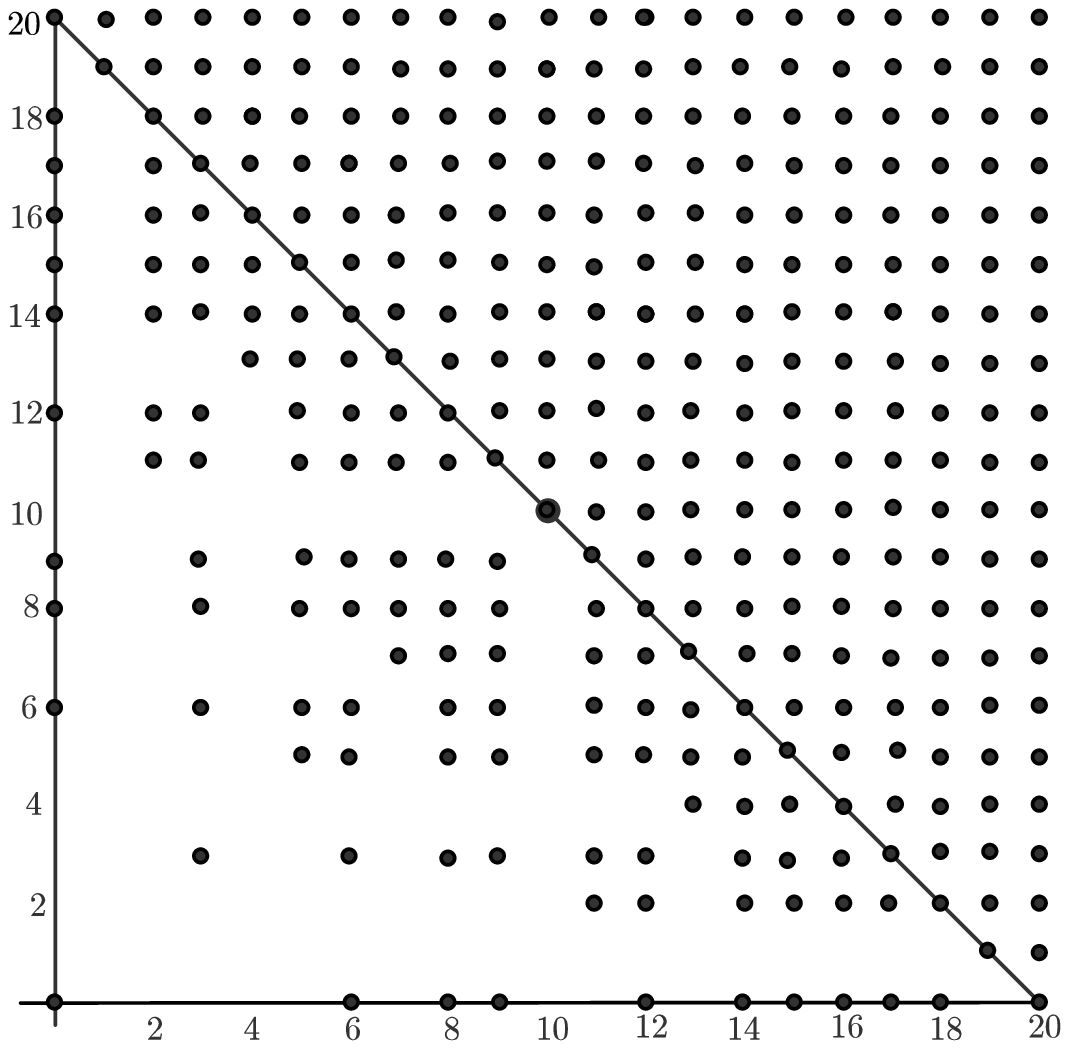}
\caption{}
\end{figure}

Now, we will determine $H(P_0, P_{\infty})$ in the case $n\geq 3$.

\medskip

\begin{lemma} \label{lemma functions}
Let $g$ be the genus of the $GK$ curve. Let $n \geq 3$ and consider the following $n^2-1$ sets of functions.\\

For $1 \leq k \leq n-1$:\\

\begin{small}
$T_k=\left\{ \dfrac{y^i z^j}{x^k} \mbox{ ; } \begin{array}{l}
 0 \leq i \leq k, \mbox{  } j \in \{ k-i+1, \ldots , n^2 - n \} \mbox{ and}\\
  k+1 \leq i \leq n, \mbox{ } j \in \{ 0,1,\ldots,n^2 - n \}
\end{array}\right \}$
\end{small}
\\

For $n \leq k \leq n^2 - n - 2$:\\

$T_k=\left\{ \dfrac{y^i z^j}{x^k} \mbox{ ; } \begin{array}{l}
0 \leq i \leq n, \mbox{  } j \in \{ k-i+1, \ldots , n^2 - n \}
\end{array} \right \}$
\\ \\

For $n^2-n-1 \leq k \leq n^2 - 1$:\\

$T_k=\left\{ \dfrac{y^i z^j}{x^k} \mbox{ ; } \begin{array}{l}
k-(n^2-1) + n \leq i \leq n, \\ j \in \{ k-i+1, \ldots , n^2 - n \}
\end{array} \right \}.$

\medskip

Let $\displaystyle T = \bigcup_{k=1}^{n^2-1} T_k$. Then, $|T|=g$.
\end{lemma}

\begin{proof}
It is not difficult to see that all the functions on $T$ are distinct. Let

$
\begin{array}{ll}
S_1  & \displaystyle = \sum_{k=1}^{n-1}|T_k | \\
                             & \displaystyle = \sum_{k=1}^{n-1} \left( \sum_{i=0}^{k} (n^2 - n - k - i) + \sum_{i=k+1}^{n} (n^2 - n +1) \right) \\
                              & \displaystyle = \sum_{k=1}^{n-1} \left( \dfrac{2n^3 - k^2 - 3k}{2} \right)\\ & = \dfrac{1}{6} (6n^4 - 7 n^3 -3n^2 + 4n);
\end{array}
$

\medskip

$
\begin{array}{ll}
S_2  & \displaystyle = \sum_{k=n}^{n^2 - n-2}|T_k | \\
                             & \displaystyle = \sum_{k=n}^{n^2-n-2} \left( \sum_{i=0}^{n} (n^2 - n - k - i) \right) \\ & = \dfrac{1}{2}(n^5 - 2n^4 - 5n - 2);
\end{array}
$

and

$
\begin{array}{ll}
S_3  & \displaystyle = \sum_{k=n^2-n-1}^{n^2 -1}|T_k | \\
                             & \displaystyle = \sum_{k=n^2-n-1}^{n^2-1} \left( \sum_{i=k-n^2+n+1}^{n} (n^2 - n - k - i) \right) \\ & = \dfrac{1}{6}(n^3 + 6n^2 + 11n + 6).
\end{array}
$

Therefore, $|T| = S_1 + S_2 + S_3 = \dfrac{1}{2} (n^5 - 2n^3 + n^2) = g$.

\end{proof}

\begin{lemma} \label{distinct divisor of poles}
Let $n \geq 3$. For each $k \in \{ 1,\ldots,n^2-1 \}$, let $T_k$ be as the previous lemma. If $f \in T_k$ and $g \in T_{k'}$, with $k,k' \in \{1,\ldots,n^2-1\}$ not necessary distinct, and $f \neq g$, then $(f)_{\infty} \neq (g)_{\infty}$.
\end{lemma}

\begin{proof}
Let $f=\dfrac{y^i z^j}{x^k} \in T_k$, with $k \in \{ 1, \ldots, n^2-1 \}$, by (\ref{divisors}) and the conditions over $i$ and $j$ follows that

\medskip
$
\left( \dfrac{y^i z^j}{x^k} \right)_{\infty}  =(k (n^3 + 1) - i (n^2 - n +1) - j) P_{0}  +(i(n^3 - n^2 + n) + j n^3 - k(n^3 + 1))P_{\infty}.
$

\medskip

We know that $ 0 \leq i\leq n$, $0 \leq j \leq n^2-n$ and $1 \leq k \leq n^2 -1$. Suppose that exists $i,i' \in \{ 0,1,\ldots, n \}$, $j,j' \in \{ 0,1,\ldots, n^2-n \}$ and $k,k' \in \{ 1,\ldots, n^2-1 \}$ such that

\medskip

$k (n^3 + 1) - i (n^2 - n +1) - j = k' (n^3 + 1) - i' (n^2 - n +1) - j'$

and

$i(n^3 - n^2 + n) + j n^3 - k(n^3 + 1) = i'(n^3 - n^2 + n) + j' n^3 - k'(n^3 + 1)$.

\medskip

Then, by the restrictions given over $i,j$ and $k$ follows that $i=i'$, $j = j'$ and $k=k'$.
\end{proof}

\medskip

\begin{proposition} \label{proposition gaps}
Let $n \geq 3$ and let $T_k$ be as the Lemma \ref{lemma functions} and $\dfrac{y^i z^j}{x^k} \in T_k$ for some $k \in \{ 1,\ldots , n^2 -1 \}$. Then,

(a) $\left( \dfrac{y^i z^j}{x^k} \right)_{\infty} = (k (n^3 + 1) - i (n^2 - n +1) - j) P_{0} + (i(n^3 - n^2 + n) + j n^3 - k(n^3 + 1))P_{\infty}$;

\medskip

(b) $(k (n^3 + 1) - i (n^2 - n +1) - j, i(n^3 - n^2 + n) + j n^3 - k(n^3 + 1)) \in G(P_0) \times G(P_{\infty})$.
\end{proposition}

\begin{proof}
(a) Follows by (\ref{divisors}) and the conditions over $i$ and $j$ on each $T_k$.

\medskip

(b) By Proposition \ref{proposition H(P)}, $H(P_0) = H(P_{\infty}) = \langle n^3 - n^2 + n , n^3 , n^3+1 \rangle$. Firstly, let us see that $k (n^3 + 1) - i (n^2 - n +1) - j \in G(P_0)$. In fact, suppose that exist $a,b,c \in \mathbb{N}_0$ such that
$$
k (n^3 + 1) - i (n^2 - n +1) - j = a(n^3 + 1) + b n^3 + c(n^3 - n^2 + n).
$$

Note that $0 \leq i \leq n$ and $0 \leq j \leq n^2-n$ and consequently we have that

\medskip

$
\left\{ \begin{array}{l}
\mbox{ if } i=n \Rightarrow a+b+c = k-1;\\
\mbox{ if } i<n \Rightarrow a+b+c = k.
\end{array} \right.
$

\medskip

Furthermore, we must have $-i(n^2-n+1) - j = -cn^2 + cn + a$ which implies

\begin{equation} \label{equation c i}
(c-i)(n^2 - n) = i+j+a-k.
\end{equation}

Now, by the conditions over $i$ and $j$ we have that $0 < i+j+a-k < n^2 - n$ (note that $a-k<0$). Thus,

$\bullet$ if $c-i\leq0$ or $c-i \geq 2$ it is easy to see that we have a contradiction in (\ref{equation c i});

$\bullet$ if $c-i=1$, then $i=c-1$ and we have $n^2-n = c-1+j+a-k$. But, for $i=n$, $c-1+j+a = j-b-2 < n^2 - n$, and for $i<n$, $c-1+a-k = j-b-1 < n^2 - n$, and we have a contradiction in both cases.

Therefore, $k (n^3 + 1) - i (n^2 - n +1) - j \notin H(P_0)$, that is $k (n^3 + 1) - i (n^2 - n +1) - j \in G(P_0)$.

Now, let us see that $i(n^3 - n^2 + n) + j n^3 - k(n^3 + 1) \in G(P_{\infty})$. As in the previous case, suppose that exist $a,b,c \in \mathbb{N}_{0}$ such that
\begin{small}
\begin{equation}\label{gaps P infinity}
i(n^3 - n^2 + n) + j n^3 - k(n^3 + 1) = a(n^3 + 1) + b n^3 + c(n^3 - n^2 + n).
\end{equation}
\end{small}
Again, by the conditions over $i,j$ and $k$ we must have

\medskip

$
\left\{ \begin{array}{l}
a+b+c = i+j-k, \mbox{ if } i<n;\\
a+b+c = n+j-k-1, \mbox{ if } i=n,
\end{array} \right.
$

\medskip

and $(c-i)(n^2 - n) = a+k$. Similarly to the previous case, follows a contradiction of equality given in (\ref{gaps P infinity}) and then $i(n^3 - n^2 + n) + j n^3 - k(n^3 + 1) \in G(P_{\infty})$.
\end{proof}

\medskip

\begin{theorem} \label{theorem Gamma}
Let $n \geq 3$, $P_0$ and $P_{\infty}$ be as above. Let

$\Gamma_{1} = \{ \gamma_{i,j,k} \mbox{ ; } 1 \leq k \leq n-1, \mbox{ } 0 \leq i \leq k, \mbox{  } k-i+1 \leq j \leq n^2 - n \}$;

$\Gamma_{2} = \{ \gamma_{i,j,k} \mbox{ ; } 1 \leq k \leq n-1, \mbox{ } k+1 \leq i \leq n, \mbox{  } 0 \leq j \leq n^2 - n \}$;

$\Gamma_{3} = \{ \gamma_{i,j,k}  \mbox{ ; } n \leq k \leq n^2-n-2, \mbox{ } 0 \leq i \leq n, \mbox{  } k-i+1 \leq j \leq n^2 - n \}$;

$\Gamma_{4} = \{ \gamma_{i,j,k}  \mbox{ ; } n^2-n-1 \leq k \leq n^2-1, \mbox{ } k-n^2+n+1 \leq i \leq n, \mbox{  } k-i+1 \leq j \leq n^2 - n \}$,

where in all sets above $\gamma_{i,j,k} = (k (n^3 + 1) - i (n^2 - n +1) - j, i(n^3 - n^2 + n) + j n^3 - k(n^3 + 1)) \in \mathbb{N}^2$. Then, $\Gamma(P_0,P_{\infty}) = \Gamma_{1} \cup \Gamma_{2} \cup \Gamma_{3} \cup \Gamma_{4}$.
\end{theorem}

\begin{proof}

By Lemma \ref{distinct divisor of poles}, we have that $(k (n^3 + 1) - i (n^2 - n +1) - j, i(n^3 - n^2 + n) + j n^3 - k(n^3 + 1)) \neq (k' (n^3 + 1) - i' (n^2 - n +1) - j', i'(n^3 - n^2 + n) + j' n^3 - k'(n^3 + 1))$ if $(i,j,k) \neq (i', j',k')$, and, by Lemma \ref{lemma functions}, we have that the numbers of triples $(i,j,k)$ is equal to genus $g$ of the curve $GK$.

Finally, by Proposition \ref{proposition gaps}, $(k (n^3 + 1) - i (n^2 - n +1) - j, i(n^3 - n^2 + n) + j n^3 - k(n^3 + 1)) \in (G(P_{\infty}) \times G(P_{0})) \cap H(P_{\infty}, P_{0})$ and the result follows by Lemma \ref{lemma 1} and the fact that the number of pairs $(k (n^3 + 1) - i (n^2 - n +1) - j, i(n^3 - n^2 + n) + j n^3 - k(n^3 + 1))$ is equal to $g$.
\end{proof}

\section{Two-Point codes on GK curve}
In this section we present two-point codes over $GK$ whose parameters are new records in the MinT's tables. In addition, we present some two-point codes that have better relative parameters when compared with certain one-point codes presented in \cite{GKcodes}.

\begin{example}
Consider the curve $GK$ with affine equations \[ Z^3=Y(1+X+X^2)\;,\qquad X^2+X=Y^3\;,\] given at the beginning of the previous section. This curve has $225$ $\mathbb{F}_{64}$-rational points and its genus is $g=10$. Remember that $H(P_0)=H(P_\infty)=\langle 6,8,9 \rangle$, $G(P_0)=G(P_\infty)=\{1,2,3,4,5,7,10,11,13,19\}$, $\Gamma(P_0, P_\infty)=$ $\{(1,19),$ $(2,11),$ (3,3),$ (4,13),$ $(5,5),$ $(7,7),$ $(10,10),$ $(11,2),$ $(13,4),(19,1)\}$ and  $H(P_{\infty},P_{0}) = \{ \mbox{lub} (\mathbf{x},\mathbf{y}) \mbox{ : } \mathbf{x},\mathbf{y} \in \Gamma(P_0, P_{\infty}) \cup (H(P_0) \times \{0\}) \cup (\{0\} \cup H(P_{\infty})) \}$.

Let $B=GK(\mathbb{F}_{64}) \setminus \{P_0,P_{\infty}\}$ and $\displaystyle D=\sum_{P\in B} P$. Consider the two-point code $C_\Omega=C_\Omega(D,(a_1+b_1-1)P_{0} + (a_2+b_2-1)P_\infty)$ of length $223$, where $a_1,a_2,b_1,b_2 >0$. If $18<\deg(G)<223$, then the dimension of the code is $k_{\Omega} = n - deg(G) + g -1 = 232- deg(G)$. Using the Theorem \ref{theorem gretchen}, we find codes $C_{\Omega}$ whose parameters are new records in MinT's tables, \cite{MinT}, in addition to improving the parameters of some codes discovered by Fanali and Guilietti in \cite{GKcodes}. The Table I show the values of $a_1,a_2,b_1$ and $b_2$ for the construction of those codes $C_{\Omega}$ with the respective dimension $k_{\Omega}$ and the bound of the minimum distance $d_{\Omega}$.

\begin{table}[h!]
\caption{}
\begin{center}
\begin{tabular}{|c|c|c|c|c|c|c|}
  \hline
  $a_1$ & $b_1$ & $a_2$ & $b_2$ & $n$ & $k_{\Omega}$ & $d_{\Omega}$ \\
  \hline
  $13$ & $10$ & $3$ & $9$ & $223$ & $199$ & $\geq 16$ \\
  $13$ & $10$ & $3$ & $10$ & $223$ & $198$ & $\geq 17$ \\
  \hline
\end{tabular}
\end{center}
\end{table}

From those two codes on the Table I we can obtain others $26$ new codes by using the Proposition \ref{proposition s}, taking $s\in \{1,2,3,\ldots 13\}$. Such codes have parameters $[223-s,199-s,\geq 16]$ and $[223-s,198-s,\geq 17]$ and all of them also improve the parameters found on MinT's tables and those codes discovered by Fanali and Guilietti in \cite{GKcodes}.
\end{example}

\medskip

For the next example let us remember that given a $[n,k,d]$-code $C$, we define its \emph{information rate} by $R=k/n$ and its \emph{relative minimum distance} by $\delta = d/n$. These parameters allows us to compare codes with different length.

\medskip

\begin{example}
Consider the curve $GK$ with affine equations \[Z^7=Y(2+X^2+2X^4+X^6)\;,\qquad X^3+X=Y^4\;.\] This curve has $6076$ $\mathbb{F}_{3^6}$-rational points and genus $g=99$. In this case $H(P_0)=H(P_\infty)=\langle 21,27,28 \rangle$ and, by Theorem \ref{theorem Gamma}, we have:

$\Gamma_1=\{(26, 26),$ $( 25, 53),$ $( 25,53),$ $( 24,80),$ $( 23,107),$ $( 22,134),$ $( 20,20),$ $( 19,47), $ $(18,74),$ $(17,101),$ $( 16,128),$ $( 15,155),$ $( 53,25),$ $( 52,52),$ $( 51,79),$ $( 50,106),$ $( 47,19),$ $( 46,46),$ $( 45,73),$ $( 44,100),$ $( 43,127),$ $( 41,13),$ $( 40,40),$ $( 39,67),$ $( 38,94),$ $( 37,121),$ $( 36,148)\}\;,$

$\Gamma_2=\{(14,14),$ $( 13 ,41),$ $( 12, 68),$ $( 11 ,95),$ $( 10 ,122),$ $( 9 1,49),$ $( 8 1,76),( 7 ,35),$ $( 6 ,62),$ $( 5 ,89),$ $( 4 ,116),$ $( 3 ,143),$ $( 2 ,170),$ $( 1 ,197),$ $( 35, 7),$ $( 34 ,34),$ $( 33 ,61),$ $( 32 ,88),$ $( 31 ,115),$ $( 30 ,142),$ $( 29 ,169)\}\;,$

$\Gamma_3=\{(80 ,24),$ $( 79 ,51),$ $( 78 ,78),$ $( 74 ,18),$ $( 73 ,45),$ $( 72 ,72),$ $( 71 ,99),$ $( 68 ,12),$ $( 67 ,39),$ $( 66 ,66),$ $( 65 ,93),$ $( 64 ,120),$ $( 62 ,6),$ $( 61 ,33),$ $( 60 ,60),$ $( 59 ,87),$ $( 58 ,114),$ $( 57 ,141),$ $( 107, 23),$ $( 106 ,50),$ $( 101 ,17),$ $( 100 ,44),$ $( 99 ,71),$ $( 95 ,11),$ $( 94 ,38),$ $( 93 ,65),$ $( 92 ,92),$ $( 89 ,5),$ $( 88 ,32),$ $( 87 ,59),$ $( 86 ,86),$ $( 85 ,113) \}\;,$

$\Gamma_4=\{(134 ,22),$ $( 128 ,16),$ $( 127 ,43),$ $( 122 ,10),$ $( 121 ,37),$ $( 120 ,64),$ $( 116 ,4),$ $( 115 ,31),$ $( 114 ,58),$ $( 113 ,85),$ $( 155 ,15),$ $( 149 ,9),$ $( 148 ,36),$ $( 143 ,3),$ $( 142 ,30),$ $( 141 ,57),$ $( 176 ,8),$ $( 170 ,2),$ $( 169 ,29),$ $( 197 ,1) \}\;.$

$\Gamma(P_0, P_\infty)=\Gamma_1\cup \Gamma_2\cup\Gamma_3\cup\Gamma_4$.

Consider the two-point code $C_\Omega=C_\Omega(D,(a_1+b_1-1)P_{0} + (a_2+b_2-1)P_\infty)$ of length $6074$, where $(a_1,a_2)=(196,1)$ and $(b_1,b_2)=(92,92-\ell)$, with $\ell\in\{0,1,\ldots, 12\}$. Then, the conditions of the Theorem \ref{theorem gretchen} are satisfied and we obtain thirteen two-point AG codes with parameters $[6074,5793-\ell,\geq 184+\ell]$. Theses thirteen two-point AG codes has relative parameters better than the one-point codes corresponding given on Table IV in \cite{GKcodes}.

\end{example}

\end{document}